\documentclass{proc-l}

\newtheorem{theorem}{Theorem}[section]
\newtheorem{lemma}[theorem]{Lemma}

\theoremstyle{definition}

\theoremstyle{remark}

\numberwithin{equation}{section}

\newcommand{\leg}[2]{\genfrac{(}{)}{}{}{#1}{#2}}
\usepackage{tikz}

\begin{document}

\title{Combinatorial Applications of M\"obius Inversion}

\author{Marie Jameson}
\address{Department of Mathematics and Computer Science, Emory University,
    Atlanta, Georgia 30322}
\email{mjames7@emory.edu}

\author{Robert P. Schneider}
\address{Department of Mathematics and Computer Science, Emory University,
    Atlanta, Georgia 30322}
\email{robert.schneider@emory.edu}


\commby{Matthew Papanikolas}

\begin{abstract}
In important work on the parity of the partition function, Ono \cite{Ono} related values of the partition function to coefficients of a certain mock theta function modulo 2. In this paper, we use M\"obius inversion to give analogous results which relate several combinatorial functions via identities rather than congruences.
\end{abstract}

\maketitle

\section{Introduction and Statement of Results}

Infinite products are ubiquitous in number theory and the theory of $q$-series. For example, recall Euler's identity
\[\prod_{n=1}^\infty(1-q^n) = \sum_{k=-\infty}^\infty (-1)^kq^{k(3k-1)/2},\] and Jacobi's identity \[\prod_{n=1}^\infty(1-q^n)^3 = \sum_{n=0}^\infty (-1)^n(2n+1)q^{n(n+1)/2}.\]

More recently, Borcherds defined ``infinite product modular forms'' \[F(z) = q^h\prod_{n=1}^\infty(1-q^n)^{a(n)},\] where $q:=e^{2\pi iz}$ and the $a(n)$'s are coefficients of certain weight 1/2 modular forms (see Chapter 4 of \cite{Ono}). This was generalized by Bruinier and Ono in \cite{BruinierOno}.

At first glance, this does not look like the stuff of combinatorics. However, one might consider the partition function $p(n)$ and ask whether the product
\begin{equation} \label{thequestion}
\prod_{n=1}^\infty (1-q^n)^{p(n)}
\end{equation}
has any special properties.  In this direction, recent work of Ono \cite{Ono}  studies the parity of $p(n).$ For $1<D\equiv 23 \pmod{24},$ Ono defined
\[\Psi_D(q):= \prod_{m=1}^\infty\prod_{0\leq b\leq D-1}\left(1-\zeta_D^{-b}q^m\right)^{\leg{-D}{b}C(\overline{m};Dm^2)},\]
where $\overline{m}$ is the reduction of $m \pmod{12},$ $\zeta_D:=e^{2\pi i/D},$ and $C(\overline{m};Dm^2)$ is the coefficient of a mock theta function. It turns out that
\[C(\overline{m};n) \equiv \begin{cases} p\left(\frac{n+1}{24}\right) \pmod{2} & \text{if } \overline{m}\equiv 1,5,7,11 \pmod{12}\\0 & \text{otherwise} \end{cases}.\]
Ono considers the logarithmic derivative
\begin{equation} \label{logderiv}
\sum_{n=1}^\infty B_D(n)q^n := \frac{1}{\sqrt{-D}}\cdot \frac{q\frac{d}{dq}\Psi_D(q)}{\Psi_D(q)} = \sum_{m=1}^\infty mC(\overline{m};Dm^2)\sum_{n=1}^\infty\leg{-D}{n}q^{mn}
\end{equation}
and notes that reducing mod 2 gives
\begin{equation} \label{reductionmod2}
\frac{1}{\sqrt{-D}}\cdot \frac{q\frac{d}{dq}\Psi_D(q)}{\Psi_D(q)} \equiv \sum_{\begin{subarray}{c} m\geq1\\ \gcd(m,6)=1 \end{subarray}}p\left(\frac{Dm^2+1}{24}\right)\sum_{\begin{subarray}{c} n\geq1\\ \gcd(n,D)=1 \end{subarray}}q^{mn}\pmod{2}.
\end{equation}

This observation was instrumental in proving strong results regarding the parity of the partition function \cite{Ono}. However, in this work we desire to establish identities rather than congruences, so it seems pertinent to again consider products of the form \eqref{thequestion}, but now at the level of $q$-series identities.

From this perspective, we wish to explore the logarithmic derivative of
\begin{equation} \label{infiniteproduct}
\prod_{n=1}^\infty(1-q^n)^{a(n)}
\end{equation}
 for other, more general combinatorial functions $a(n).$ Then for a nonnegative integer $n$, define
\begin{align*}
Q(n) &:= \# \text{of partitions of $n$ into distinct parts}\\
\widehat{Q}(n) &:= \# \text{ of partitions of $n$ whose parts occur with the same multiplicity}
\end{align*} 
and
\begin{align*}
F_Q(q) &:=\sum_{n=1}^\infty Q(n)q^n\\
F_{\widehat{Q}}(q) &:= \sum_{n=1}^\infty \widehat{Q}(n)q^n\\
\Psi(Q;q) &:=\prod_{n=1}^\infty (1-q^n)^{Q(n)/n}.
\end{align*} 

\begin{theorem} \label{thm1}
We have that \[\frac{q\frac{d}{dq}\Psi(Q;q)}{\Psi(Q;q)} = -F_{\widehat{Q}}(q).\] Moreover, for all $n\geq 1$ we have \[Q(n) = \sum_{d|n}\mu(d)\widehat{Q}(n/d),\] where $\mu$ denotes the M\"obius function.
\end{theorem}

For example, one can compute that 
\begin{align*}
\Psi(Q;q) &=1-q-\frac{1}{2}q^2-\frac{1}{6}q^3+\frac{1}{24}q^4+\frac{43}{120}q^5-\frac{233}{720}q^6+\cdots\\
\frac{q\frac{d}{dq}\Psi(Q;q)}{\Psi(Q;q)}  &= -q-2q^2-3q^3-4q^4-4q^5-8q^6-\cdots\\
F_{\widehat{Q}}(q) &= q+2q^2+3q^3+4q^4+4q^5+8q^6+\cdots = -\frac{q\frac{d}{dq}\Psi(Q;q)}{\Psi(Q;q)}.
\end{align*}

In fact, while it is not obvious from a combinatorial perspective, this theorem is simple; it follows from the straightforward observation that \[\widehat{Q}(n) = \sum_{d|n}Q(d).\] Now we present two results in a slightly different direction that are perhaps more surprising. Looking again to the work of Ono \cite{Ono}, we can apply M\"obius inversion to \eqref{logderiv} to find
\begin{equation}\label{mobiusinv}
C(\overline{n};Dn^2) = \frac{1}{n}\sum_{d|n} \mu(d)\leg{-D}{d}B_D(n/d).
\end{equation}
It is natural to ask whether there are analogs of this statement for related $q$-series, even if the series do not arise as logarithmic derivatives of Borcherds products.

We begin our search of interesting combinatorial functions by noting that the generating function for the partition function $p(n)$ obeys the identity of Euler
\[P(q):=\sum_{n=0}^\infty p(n)q^n = \sum_{n=0}^\infty \frac{q^{n^2}}{(q)_n^2}\]
where $(q)_n$ is the Pochhammer symbol, defined by $(q)_0=1$ and $(q)_n=\prod_{k=1}^n (1-q^k)$ for $n\geq1$. We wish to investigate other functions of a similar form, such as those presented in the following theorems, which are formally analogous to \eqref{mobiusinv} but involving other combinatorial functions. 

Let $p_a(n)$ denote the number of partitions of $n$ into $a$ parts, and define $\widehat{p}_a(n)$ to be the number of partitions of $n$ into $ak$ parts for some integer $k\geq 1$, i.e.  \[\widehat{p}_a(n) := \sum_{j=1}^\infty p_{aj}(n).\]  On analogy to the identities for $P(q)$ above, we let $P_a(q)$ and $\widehat{P}_a(q)$ denote the generating functions of $p_a(n)$ and $\widehat{p}_a(n),$ respectively. Then we have the following identities for $P_a(q)$ and $\widehat{P}_a(q)$.
\begin{theorem}\label{thm2} We have that
\begin{align*}
P_a(q) &=\sum_{n=1}^\infty{\mu (n)\widehat{P_{an}}(q)}\\
p_a(n) &= \sum_{j=1}^\infty {\mu (j)\widehat{p_{aj}}(n)}.
\end{align*} 
\end{theorem}

Observe that for $a=1,$ we have that $p_1(n)=1$ for all integers $n$, and also that
\[\widehat{p_1}(n)= \sum_{j=1}^\infty {p_j(n)}=p(n).\] 
In this case, the generating functions are given by
\[P_1(q) = \sum_{n=1}^\infty p_1(n)q^n = \sum_{n=1}^\infty q^n = \frac{q}{1-q}\] 
and
\[\widehat{P_1}(q)=\sum_{n=1}^\infty{\widehat{p_1}(n)q^n}=\sum_{n=1}^\infty{p(n)q^n}.\] 
Thus by Theorem \ref{thm2}, we have the explicit identities
\[P_1(q)=\sum_{n=1}^\infty \mu(n)\widehat{P_n}(q) =\frac{q}{1-q}\] 
and, perhaps more interestingly,
\[\ p_1(n)= \sum_{j=1}^\infty {\mu (j)\widehat{p_j}(n)}=1.\]

Looking again for identities similar to those given above for $P(q),$ for a positive integer $a$ set
\begin{align*}
B_a(q) &:= \sum^{\infty }_{n=1}{\frac{q^{n^2+an}}{(q)^2_n}}=:\sum^{\infty }_{N=1}{b_a(N)q^N}\\
\widehat{B}_a(q) &:= \sum^{\infty }_{n=1}{\frac{q^{n^2+an}}{(q)^2_n\left(1-q^{an}\right)}}=:\sum^{\infty }_{N=1}{\widehat{b}_a(N)q^N}.
\end{align*}

Generalizations of $q$-series such as $B_a(q)$ and $\widehat{B}_a(q)$ have been studied by Andrews \cite{Andrews}.  One can give a combinatorial interpretation for the coefficients $b_a(N)$ and $\widehat{b}_a(N)$ as follows.

Consider the Ferrers diagram of a given partition of an integer $N$ with an $n\times n$ Durfee square, and having a rectangle of base $n$ and height $m$ adjoined immediately below the $n\times n$ Durfee square. For example, the partition of $N=12$ shown below has a $2\times 2$ Durfee square (marked by a solid line), and either a $2\times 2$ or $2\times 1$ rectangle below it (the $2\times 1$ rectangle is marked by a dashed line).  
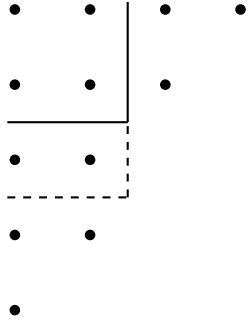
\vspace{0.5cm}
\begin{center}
\begin{tikzpicture}[inner sep=0pt,thick,
    dot/.style={fill=black,circle,minimum size=4pt}]
\node[dot] (a) at (0,0) {};
\node[dot] (a) at (1,1) {};
\node[dot] (a) at (0,1) {};
\node[dot] (a) at (0,2) {};
\node[dot] (a) at (1,2) {};
\node[dot] (a) at (0,3) {};
\node[dot] (a) at (1,3) {};
\node[dot] (a) at (0,4) {};
\node[dot] (a) at (1,4) {};
\node[dot] (a) at (2,4) {};
\node[dot] (a) at (3,4) {};
\node[dot] (a) at (2,3) {};
\draw[-] (-0.1,2.5)--(1.5,2.5);
\draw[-] (1.5,2.5)--(1.5,4.1);
\draw[dashed] (-0.1,1.5)--(1.5,1.5);
\draw[dashed] (1.5,1.5)--(1.5,2.5);
\end{tikzpicture}
\end{center}
We refer to this rectangular region of the diagram as an $n\times m$ ``Durfee rectangle,'' and note that a given Ferrers diagram may have nested Durfee rectangles of sizes $n\times 1, n\times 2, \ldots, n\times M$, where $M$ is the height of the largest such rectangle (assuming that at least one Durfee rectangle is present in the diagram).  

We then have that
\begin{align*}
b_a(N)=& \# \text{ of partitions of } N \text{ having an } n\times n \text{ Durfee square and at least an } n\times a \text{ Durfee}\\ & \text{rectangle}\\
\widehat{b_a}(N) =&\# \text{ of partitions of } N \text{ having an } n\times n \text{ Durfee square and at least an } n \times a \text{ Durfee}\\
& \text{rectangle (counted with multiplicity as an } n \times a \text{ rectangle may be nested within }\\
& \text{taller Durfee rectangles of size } n\times ak, \text{ for } k\geq1).
\end{align*}
Assuming these notations, we have the following result.

\begin{theorem} \label{thm3} We have that
\[\widehat{b_a}(n)= \sum_{j=1}^\infty {b_{aj}(n)}.\]
Moreover, we have
\begin{align*}
B_a(q)&=\sum_{n=1}^\infty{\mu (n)\widehat{B_{an}}(q)}\\
b_a(n)&= \sum_{j=1}^\infty {\mu (j)\widehat{b_{aj}}(n)}.
\end{align*}
\end{theorem}


\section{Proof of Theorem \ref{thm1}}
First we prove a lemma regarding logarithmic derivatives.

\begin{lemma}\label{lem}
For any sequence $\{a(n)\},$ we have that
\[\frac{q\frac{d}{dq}\left(\prod_{n=1}^\infty(1-q^n)^{a(n)}\right)}{\prod_{n=1}^\infty(1-q^n)^{a(n)}} = - \sum_{n=1}^\infty \sum_{d|n}a(d)dq^n.\]
\end{lemma}
\begin{proof}
Since $\log(1-x) = -\sum_{m=1}^\infty \frac{x^m}{m},$ we have that
\begin{align*}
\frac{q\frac{d}{dq}\left(\prod_{n=1}^\infty(1-q^n)^{a(n)}\right)}{\prod_{n=1}^\infty(1-q^n)^{a(n)}} &= q\frac{d}{dq}\left(\log\left(\prod_{n=1}^\infty(1-q^n)^{a(n)}\right)\right) = q\frac{d}{dq}\left(\sum_{n=1}^\infty a(n)\log\left(1-q^n\right)\right)\\
&= -q\frac{d}{dq}\left(\sum_{n=1}^\infty a(n)\sum_{m=1}^\infty \frac{q^{mn}}{m}\right) = -\left(\sum_{n=1}^\infty a(n)\sum_{m=1}^\infty nq^{mn}\right)\\
&= -\sum_{n=1}^\infty\sum_{d|n}a(d)dq^n
\end{align*}
as desired.
\end{proof}

\begin{proof}[Proof of Theorem \ref{thm1}]
First note that for all $n\geq 1$ we have \[\widehat{Q}(n) =\sum_{d|n}Q(d),\] so $Q(n) = \sum_{d|n}\mu(d)\widehat{Q}(n/d)$ by M\"obius inversion. By Lemma \ref{lem}, we have that \[\frac{q\frac{d}{dq}\Psi(Q;q)}{\Psi(Q;q)} = - \sum_{n=1}^\infty \sum_{d|n}Q(d)q^n = - \sum_{n=1}^\infty \widehat{Q}(n)q^n\] as desired.
\end{proof}

\section{Proof of Theorems \ref{thm2} and \ref{thm3}}

Suppose that for each positive integer $a$, we have two arithmetic functions $f(a;n)$ and $\widehat{f}(a;n)$ such that \[\widehat{f}(a;n) = \sum_{j=1}^\infty f(aj;n),\] where the above sum converges absolutely. We will define their generating functions as follows.
\begin{align*}
F(a;q) &:= \sum_{n=1}^\infty f(a;n)q^n\\
\widehat{F}(a;q) &:= \sum_{n=1}^\infty\widehat{f}(a;n)q^n.
\end{align*}
We then have the following result.
\begin{lemma} \label{lemma2} We have that
\[F(a;q)=\sum_{n=1}^\infty\mu (n)\widehat{F}(an;q)\]
and
\[f(a;n)= \sum_{j=1}^\infty \mu(j)\widehat{f}(aj;n).\]
\end{lemma}
 
\begin{proof}
Recall that \[\sum_{d|n}\mu(n) = \begin{cases} 1 & \text{if } n=1\\ 0 & \text{otherwise}\end{cases}.\] It follows that
\begin{align*}
F\left(a;q\right) &=\sum_{n=1}^\infty\left(\sum_{k\ge 1} f(an;k)q^k\right) \sum_{d|n} \mu(d)\\
&=\sum_{n=1}^\infty\mu(n) \sum_{k\ge 1} \left( \sum_{j=1}^\infty  f(anj;k)\right)q^k\\
&=\sum_{n=1}^\infty\mu(n) \sum_{k\ge 1}\widehat{f}(an;n)q^k\\
&=\sum_{n=1}^\infty\mu(n) \widehat{F}(an;q).
\end{align*}
Then by comparing coefficients, one finds that $f(a;n) = \sum_{j=1}^\infty \widehat{f}(aj;n),$ as desired.
\end{proof}

This lemma can be used to prove both Theorem \ref{thm2} and Theorem \ref{thm3}. We note that Lemma \ref{lemma2} can be applied in extremely general settings, and one has great freedom in creatively choosing the constant $a$ to be varied. For instance, taking $a=1$ gives rise to any number of identities, as $1$ can be inserted as a factor practically anywhere in a given expression.   

\begin{proof}[Proof of Theorem \ref{thm2}] The theorem follows by a direct application of Lemma \ref{lemma2}.
\end{proof}

\begin{proof}[Proof of Theorem \ref{thm3}] First note that \[\widehat{b}_a(N) = \sum_{j=1}^\infty b_{aj}(N),\] since
\begin{align*}
\widehat{B}_a(q) &= \sum^{\infty }_{n=1}\frac{q^{n^2+an}}{(q)^2_n\left(1-q^{an}\right)} = \sum_{n=1}^\infty \frac{q^{n^2+an}}{(q)^2_n} \sum_{j=0}^\infty q^{ajn}\\
&= \sum_{j=1}^\infty \sum_{n=1}^\infty \frac{q^{n^2+ajn}}{(q)^2_n} = \sum_{j=1}^\infty B_{aj}(q).
\end{align*}
The rest follows by applying Lemma \ref{lemma2}.
\end{proof}

\section*{Acknowledgements}

The authors thank Ken Ono, Robert Lemke Oliver and Andrew Granville for their useful comments and insights.

\bibliographystyle{amsplain}

\begin{thebibliography}{1}

\bibitem{Andrews}
George~E. Andrews.
\newblock Concave compositions.
\newblock {\em Electron. J. Combin.}, 18(2):Paper 6, 13, 2011.

\bibitem{Apostol}
Tom~M. Apostol.
\newblock {\em Introduction to analytic number theory}.
\newblock Springer-Verlag, New York, 1976.
\newblock Undergraduate Texts in Mathematics.

\bibitem{BruinierOno}
Jan~H. Bruinier and Ken Ono.
\newblock The arithmetic of {B}orcherds' exponents.
\newblock {\em Math. Ann.}, 327(2):293--303, 2003.

\bibitem{Onobook}
Ken Ono.
\newblock {\em The web of modularity: arithmetic of the coefficients of modular
  forms and {$q$}-series}, volume 102 of {\em CBMS Regional Conference Series
  in Mathematics}.
\newblock Published for the Conference Board of the Mathematical Sciences,
  Washington, DC, 2004.

\bibitem{Ono}
Ken Ono.
\newblock Parity of the partition function.
\newblock {\em Adv. Math.}, 225(1):349--366, 2010.

\end{thebibliography}

\end{document}